\documentclass[reqno]{amsart}
\usepackage{graphicx}
\usepackage{pictex}
\usepackage{mathrsfs}
\usepackage{hyperref}
\begin{document}

\newtheorem{thm}{Theorem}[section]
\newtheorem{lem}[thm]{Lemma}
\newtheorem{cor}[thm]{Corollary}
\newtheorem{add}[thm]{Addendum}
\newtheorem{prop}[thm]{Proposition}
\theoremstyle{definition}
\newtheorem{defn}[thm]{Definition}
\theoremstyle{remark}
\newtheorem{rmk}[thm]{Remark}
\newcommand{\SLtwoC}{\mathrm{SL}(2,{\mathbf C})}
\newcommand{\SLtwoR}{\mathrm{SL}(2,{\mathbf R})}
\newcommand{\PSLtwoC}{\mathrm{PSL}(2,{\mathbf C})}
\newcommand{\PSLtwoR}{\mathrm{PSL}(2,{\mathbf R})}
\newcommand{\SLtwoZ}{\mathrm{SL}(2,{\mathbf Z})}
\newcommand{\PSLtwoZ}{\mathrm{PSL}(2,{\mathbf Z})}
\newcommand{\CmodTwoPiIZ}{{\mathbf C}/2\pi i {\mathbf Z}}
\newcommand{\MCG}{\mathcal{MCG}}
\newcommand{\Cnozero}{{\mathbf C}\backslash \{0\}}
\newcommand{\Cinfty}{{\mathbf C}_{\infty}}
\newcommand{\HH}{H^2}
\newcommand{\HHH}{H^3}
\newcommand{\tr}{{\hbox{tr}}}
\newcommand{\Hom}{\mathrm{Hom}}
\newcommand{\SL}{\mathrm{SL}}
\newenvironment{pf}{\noindent {\it Proof.}\quad}{\square \vskip 10pt}

\title[A new identity for $\SLtwoC$-characters]{A new identity for $\SLtwoC$-characters of the once punctured torus group}
\author{Hengnan Hu, Ser Peow Tan, and Ying Zhang}
\address{Department of Mathematics \\ National University of Singapore \\ Singapore 119076} \email{huhengnan@nus.edu.sg}
\address{Department of Mathematics \\ National University of Singapore \\ Singapore 119076} \email{mattansp@nus.edu.sg}
\address{School of Mathematical Sciences \\ Soochow University \\ Suzhou 215006 \\ China} \email{yzhang@suda.edu.cn}

\dedicatory{To Professor Sadayoshi Kojima on the occasion of his
sixtieth birthday}


\thanks{
 The second author is partially supported by the National University of Singapore academic research grant R-146-000-186-112.
 The third author is supported by NSFC (China) grant no. 11271276 and
 Ph.D. Programs Foundation (China) grant no. 20133201110001. }

%
%

 \begin{abstract}
 We obtain new variations of the original McShane identity for those $\SLtwoC$-representations of the once punctured torus group
 which satisfy the Bowditch conditions, and also for those fixed up to conjugacy by an Anosov mapping class of the torus
 and satisfying the relative Bowditch conditions.
 \end{abstract}

 \maketitle




 \vspace{10pt}

 \section{{\bf Introduction}}\label{s:intro}

 For a once punctured torus $T$ equipped with any complete hyperbolic structure of finite area,
 Greg McShane in \cite{mcshane1991thesis} obtained the original McShane identity:
 \begin{eqnarray}\label{eqn:mcshane}
 \sum_{\gamma}\frac{1}{1+e^{l(\gamma)}}=\frac{1}{2},
 \end{eqnarray}
 where the sum is over all simple closed geodesics $\gamma$ in $T$, with $l(\gamma)$ the length of $\gamma$.

 The identity (\ref{eqn:mcshane}) has since been generalized to hyperbolic surfaces with cusps (McShane \cite{mcshane1998im})
 or smooth geodesic boundary (Mirzakhani \cite{mirzakhani2007im}) or conic singularities (Tan-Wong-Zhang \cite{tan-wong-zhang2006jdg}),
 and eventually from a different point of view to closed surfaces (Luo-Tan \cite{luo-tan}),
 along with other related versions (for example, Rivin \cite{rivin2012advm}).

 On the other hand, Brian Bowditch gave in \cite{bowditch1996blms} a simple proof of identity (\ref{eqn:mcshane}) via Markoff triples,
 and extended it in \cite{bowditch1998plms} to type-preserving representations of the once punctured torus group
 into ${\rm SL}(2, \mathbf C)$ satisfying certain conditions (which we call the Bowditch conditions).
 In particular, he showed identity (\ref{eqn:mcshane}) holds for quasi-Fuchsian representations of the once punctured torus group.
 S. P. Tan, Y. L. Wong and Y. Zhang in \cite{tan-wong-zhang2008advm} generalized the McShane-Mirzakhani identity to an identity
 for general irreducible representations of the once punctured torus group into ${\rm SL}(2, \mathbf C)$ satisfying the same set of Bowditch conditions.

 In this paper we obtain new identities for general irreducible representations of the once punctured torus group into ${\rm SL}(2, \mathbf C)$
 satisfying the Bowditch conditions.

 It is well known that a complete hyperbolic structure on $T$ gives rise to (the conjugacy class of) a discrete, faithful representation of $\pi_1(T)$ into $\PSLtwoR$, the group of orientation-preserving isometries of the upper half-plane model of the hyperbolic plane $\HH$.
 Since $\pi_1(T)$ is a free group (of rank two), we may lift the representations into $\SLtwoR$ and simply consider representations of $\pi_1(T)$ into $\SLtwoR$.
 In general, to study  deformations of related  hyperbolic $3$-manifolds, we consider the representation set ${\rm Hom}\,(\pi_1(T), \SLtwoC)$.

 We say that a representation $\rho: \pi_1(T) \rightarrow \SLtwoC$ has {\it peripheral trace} $\tau \in {\mathbf C}$ if
 $$
 {\rm tr}\,\rho(aba^{-1}b^{-1}) = \tau
 $$
 where $\{a,b\}$ is a free basis of $\pi_1(T)$.
 It can be shown that $\tau$ is independent of the choice of the free basis,
 and that $\rho$ is reducible if and only if it has peripheral trace $\tau=2$.
 In the special case where $\tau=-2$, $\rho$ is said to be {\it type-preserving},
 as $\rho(aba^{-1}b^{-1})\in\SLtwoC$ is parabolic unless ${\rm tr}\,\rho(a)={\rm tr}\,\rho(b)=0$.

 Any two representations $\rho, \rho': \pi_1(T) \rightarrow \SLtwoC$ which are conjugate by an element in $\SLtwoC$ have the same trace values.
 Conversely, two irreducible representations $\rho, \rho': \pi_1(T) \rightarrow \SLtwoC$ with the same trace values are conjugate by an element in $\SLtwoC$.
 In fact, from the trace values of $\rho$ on a triple $a,b,ab$ where $\{a,b\}$ is a free basis of $\pi_1(T)$,
 W. Goldman \cite[page 451]{goldman2003gt} obtained a very simple representative $\rho$ of its conjugacy class.

 Let $\Omega$ be the set of isotopy classes of (unoriented) essential (i.e., nontrivial and not boundary-parallel) simple closed curves in $T$.
 Thus an element $X \in \Omega$ corresponds to the union of two conjugacy classes of a pair of inverse elements in $\pi_1(T)$;
 in particular, the trace ${\rm tr}\,\rho(X) \in {\mathbf C}$ is well defined.

 An irreducible representation $\rho: \pi_1(T) \rightarrow \SLtwoC$ is said to satisfy the {\it Bowditch conditions} if
 (i) ${\rm tr}\,\rho(X) \not\in [-2,2] \subset {\mathbf R}$ for all $X \in \Omega$, and
 (ii) the set $\{ X \in \Omega \,:\, |\tr\,\rho(X)| \le 2\}$ is finite (possibly empty).

 As the main theorem of this paper, we have

 \begin{thm}\label{thm:main}
 Let $\rho: \pi_1(T) \rightarrow \SLtwoC$ be an irreducible representation
 and let $\mu=\tau+2$ where $\tau \in \mathbf C$ is the peripheral trace of $\rho$.
 If $\rho$ satisfies the Bowditch conditions then
 \begin{equation}\label{eqn:main}
 \sum_{X \in \Omega} h_{\mu}(x) = \frac{1}{2},
 \end{equation}
 where the infinite sum converges absolutely, $x={\rm tr}\,\rho(X)$, and the function
 $h_{\mu}:{\mathbf C}\backslash \{[-2,2],\pm \sqrt{\mu}\} \rightarrow {\mathbf C}$ is defined by
 \begin{equation}\label{eqn:hmu}
  h_{\mu}(x) = \frac{1}{2} \left(1- \frac{x^2-\frac23\mu}{x^2-\mu} \sqrt{1-\frac{4}{x^2}}\right).
 \end{equation}
 \end{thm}

 Here and throughout the paper, we assume that, the square root function be taken as
 $\sqrt{\phantom{+}}: {\mathbf C} \rightarrow \{ z\in {\mathbf C} \mid \Re(z) \ge 0 \}$.
 Note that identity (\ref{eqn:main}) when $\mu=0$ reduces to Bowditch's extension of the original McShane identity
 (\ref{eqn:mcshane}).

 If we consider conjugacy classes of representations $\pi_1(T) \rightarrow \SLtwoC$ which are fixed by an Anosov mapping class $\theta$ of $T$,
 then the Bowditch conditions are not satisfied;
 but we can define the relative Bowditch conditions on $\Omega/\theta$, that is, if the corresponding conditions (i) and (ii)
 are satisfied by the elements of $\Omega/\theta$. As in \cite{bowditch1997t}, we have the following variant of Theorem \ref{thm:main}.

 \begin{thm}\label{thm:relmain}
 Let $\rho: \pi_1(T) \rightarrow \SLtwoC$ be an irreducible representation and let $\mu=\tau+2$ where $\tau \in \mathbf C$ is the peripheral trace of $\rho$.
 If $\rho$ is fixed up to conjugation by an Anosov mapping class $\theta$ of $T$ and it satisfies the Bowditch conditions relative to the action of $\theta$
 then
 \begin{equation}\label{eqn:relmain}
 \sum_{[X] \in \Omega/\theta} h_{\mu}({\rm tr}\,\rho(X)) = 0,
 \end{equation}
 where the infinite sum converges absolutely, with function $h_{\mu}$ defined by (\ref{eqn:hmu}).
 \end{thm}

 There is an identification of $\Omega$ with $\mathbf Q \cup \{1/0\}$ by considering the slopes of $\mathbf Z$-homology classes of simple closed curves in $T$.
 Explicitly, given a free basis $\{a,b\}$ of $\pi_1(T)$, if the simple closed curves represented by (the conjugacy classes of) $a$, $b$ and $ab$ are
 assumed to have slopes $0/1$, $1/0$ and $1/1$, then the simple closed curve represented by $ab^{-1}$ has slope $-1/1$,
 and each essential simple closed curve in $T$ has a unique slope $s/r \in \mathbf Q \cup \{1/0\}$ where $r,s$ are coprime integers.

 Furthermore, there is a tri-coloring of $\Omega$ coming from the three non-trivial ${\mathbf Z}_2$-homology classes,
 or equivalently, from considering the identification of $\Omega$ with $\mathbf Q \cup \{1/0\}$, and the parity of $s/r \in \mathbf Q \cup \{1/0\}$,
 namely, when only $r$ is odd, or when only $s$ is odd, or when $r,s$ are both odd.
 Thus
 $$
 \Omega= \Omega_1 \sqcup \Omega_2 \sqcup \Omega_3.
 $$

 Theorem \ref{thm:main} is a special case (with $p_i=\frac{1}{3}$) of the following theorem.

 \begin{thm}\label{thm:mainpqr}
 Let $\rho: \pi_1(T) \rightarrow \SLtwoC$ be an irreducible representation and let $\mu=\tau+2$
 where $\tau \in {\mathbf C}$ is the peripheral trace of $\rho$.
 Choose arbitrary $p_1,p_2,p_3 \in {\mathbf C}$ such that $p_1+p_2+p_3=1$.
 If $\rho$ satisfies the Bowditch conditions then
 \begin{equation}\label{eqn:mainpqr}
 \sum_{X \in \Omega_1} h_{\mu,p_1}(x)+\sum_{X \in \Omega_2} h_{\mu,p_2}(x)+\sum_{X \in \Omega_3} h_{\mu,p_3}(x)= \frac{1}{2}, %
 \end{equation}
 where the three infinite sums converge absolutely,
 $x={\rm tr}\rho(X)$ and the function $h_{\mu,p}:{\mathbf C}\backslash\{[-2,2],\pm\sqrt{\mu}\} \rightarrow {\mathbf C}$ is defined by
 \begin{equation}
 h_{\mu,p}(x)=\frac{1}{2}\left(1- \frac{x^2-(1-p)\mu}{x^2-\mu} \sqrt{1-\frac{4}{x^2}}\right).
 \end{equation}
 \end{thm}

 \noindent {\bf Remark.} Most of the results here generalize to an $n$-variable setting where we consider the polynomial automorphisms
 of ${\mathbf C}^n$ which preserve the Markoff-Hurwitz equation, see \cite{hu-tan-zhang2} or \cite{hu2013thesis} for details.
 We also note here that while the identities in \cite{tan-wong-zhang2008advm} were geometrically motivated, these are not,
 hence it would be interesting to find some geometric interpretations for these new identities.

 \medskip

 The rest of the paper is organized as follows.
 In \S\ref{s:prelim} we describe the combinatorial setting and formulation introduced by Bowditch in \cite{bowditch1998plms}
 and state a branch version of Theorem \ref{thm:main}.
 In \S \ref{s:proofs} we prove the main theorem, and indicate proofs of the other results.
 We also give a second proof of Theorem \ref{thm:mainpqr}, which explains how the result was originally motivated.
 In \S \ref{s:appendix} we give two summation identities used in the second proof of Theorem \ref{thm:mainpqr}.

 \medskip

 {\it Acknowledgements.} We would like to thank Martin Bridgeman, Dick Canary, Bill Goldman, Fran\c{c}ois Labourie, Makoto Sakuma  and
 Weiping Zhang for helpful conversations and comments.

 \section{Preliminary setting and formulations}\label{s:prelim}

 In this section we describe the setting and formulation introduced by Bowditch in \cite{bowditch1998plms} (and used in \cite{tan-wong-zhang2008advm})
 on the combinatorial structure of the set $\Omega$ of isotopy classes of essential simple closed curves in a once punctured torus $T$.

 The curve complex ${\mathscr C}(T)$ of $T$ is a simplicial complex defined as follows:
 the set of vertices ($0$-simplices) is $\Omega$; there is an edge ($1$-simplex) with vertex set $\{X,Y\}$
 if and only if the two simple closed curves representing $X$ and $Y$ have minimal geometric intersection number one;
 and there is a $k$-simplex with vertex set $\{X_0,X_1,\cdots,X_k\}$ if and only if each pair of the vertices is connected by an edge.

 The simplicial complex ${\mathscr C}(T)$ is $2$-dimensional.
 Actually, ${\mathscr C}(T)$ has a nice geometric realization in the upper half-plane model of the hyperbolic plane $\HH$ via the Farey triangulation:
 the set of vertices is ${\mathbf Q}\cup\{1/0\}$, the set of rational slopes (including $1/0$);
 the edges of ${\mathscr C}(T)$ are realized as hyperbolic lines joining Farey neighbors;
 and the $2$-simplices are the ideal triangles of the Farey triangulation.

 \vskip 5pt

 \textbf{Trivalent tree $\Sigma$.} The dual graph $\Sigma$ of ${\mathscr C}(T)$ is a trivalent tree embedded in the underlying space $|{\mathscr C}(T)|$.
 Geometrically, $\Sigma$ can be realized as a trivalent tree properly embedded in $\HH$:
 the set of vertices, $V(\Sigma)$, consists of the geometric centers of the rational ideal triangles,
 and the set of edges, $E(\Sigma)$, consist of geodesic segments.
 A complementary region of $\Sigma$ is the closure of a connected component of the complement of $\Sigma$ in $\HH$.
 The set of complementary regions of $\Sigma$ in $\HH$ has a natural bijection with $\mathbf{Q}\cup \{1/0\}$ which is identified with $\Omega$.
 We thus identify the two sets and also use $\Omega$ to denote the set of complementary regions of $\Sigma$ in $\HH$.

 For every vertex $v \in V(\Sigma)$, we use the identification $v \leftrightarrow \{X,Y,Z\}$ to indicate the
 three complementary regions $X,Y,Z \in \Omega$ around $v$, that is, $v=X \cap Y \cap Z$.
 For every edge $e \in E(\Sigma)$, we use the notation $e \leftrightarrow \{X,Y;Z,W\}$ to mean that
 $e=X \cap Y$ and $e \cap Z$ and $e \cap W$ are endpoints of $e$; we also use $\Omega^0(e):=\{X,Y\}$.

 \vskip 5pt

 \textbf{Directed edges of $\Sigma$ and circular sets.}
 Denote by $\vec{E}(\Sigma)$ the set of all directed edges of $\Sigma$.
 Each edge $e \in E(\Sigma)$ corresponds to a pair of oppositely directed edges $\vec{e}, -\vec{e} \in \vec{E}(\Sigma)$.
 For a directed edge $\vec{e} \in \vec{E}(\Sigma)$, we use the notation $\vec{e} \leftrightarrow \{X,Y;Z \to W\}$,
 or just $\{X,Y;\to W\}$ for simplicity, to indicate that $e \leftrightarrow \{X,Y;Z,W\}$ and the
 direction of $\vec{e}$ heads towards $W$. Thus we have $-\vec{e}\leftrightarrow \{X,Y;W\to Z\}$.

 Let $\Omega^+(\vec e)$ (resp., $\Omega^-(\vec e)$) be the subset of $\Omega$ consisting of those complementary regions of $\Sigma$
 whose boundaries are completely contained in the subtree of $\Sigma$ obtained by removing $e$ from $\Sigma$ that contains the head (resp., tail) of $\vec e$.
 Thus we have, for every $\vec e \in {\vec E}(\Sigma)$, $\Omega =\Omega^-(\vec e) \sqcup \Omega^0(e) \sqcup \Omega^+(\vec e)$.

 Given a finite subtree $\Sigma'$ of $\Sigma$, the {\it circular set} $C(\Sigma')\subset\vec{E}(\Sigma)$ is defined
 to be the set of all $\vec{e} \in \vec{E}(\Sigma)$ such that $e \cap \Sigma'$ consists of only the head endpoint of $\vec{e}$.

 \vskip 5pt

 \textbf{$\SLtwoC$-characters and Markoff maps.}
 Given a representation $\rho:\pi_1(T)\rightarrow \SLtwoC$, its trace function ${\rm tr}\rho: \pi_1(T)\rightarrow {\mathbf C}$
 is called an $\SLtwoC$-{\it character} of the once punctured torus group.
 Furthermore, ${\rm tr}\rho$ regarded as defined on $\Omega$ is called a $\mu$-{\it Markoff map}, where $\mu=\tau+2$
 with $\tau \in \mathbf C$ the peripheral trace of $\rho$.

 Equivalently, a $\mu$-Markoff map is a function $\phi: \Omega \to {\mathbf C}$
 satisfying the following vertex and edge relations (writing $x:=\phi(X)$, $y:=\phi(Y)$, etc.):

 (i) for every vertex $v \leftrightarrow \{X, Y, Z\}$, $(x,y,z)$ satisfy the $\mu$-Markoff equation
 \begin{equation}\label{eqn:mumarkoff}
 x^2 +y^2 + z^2 - xyz = \mu;
 \end{equation}

 (ii) for every edge $e \leftrightarrow \{X,Y;Z,W\}$, $(x,y,z,w)$ satisfy the edge equation
 \begin{equation}\label{eqn:edge}
 xy = z + w.
 \end{equation}

 \textbf{Edge direction induced by a Markoff map.}
 There is a natural direction on edges of $\Sigma$ determined by a Markoff map $\phi$ as follows:
 for $e \leftrightarrow \{X,Y;Z,W\}$,
 if $|\phi(W)| > |\phi(Z)|$, we choose $\vec{e} \leftrightarrow \{X,Y;W \to Z \}$;
 if $|\phi(W)| < |\phi(Z)|$, we choose $\vec{e} \leftrightarrow \{X,Y;Z \to W \}$;
 if $|\phi(W)| = |\phi(Z)|$, then choose the direction arbitrarily.

 \vskip 5pt

 \textbf{Edge value defined by a Markoff map.}
 For a directed edge $\vec{e} \leftrightarrow \{X,Y; \to Z\}$, writing $x=\phi(X)$, $y=\phi(Y)$, $z=\phi(Z)$,
 we define an edge value $\phi(\vec{e})$ by
 \begin{equation}\label{eqn:phivece}
 \phi(\vec{e}) =\Psi(x,y;z)= \frac{z}{xy} - \frac{1}{3} \left(\frac{1}{2} - \frac{z}{xy}\right)\left(\frac{\mu}{x^2-\mu} + \frac{\mu}{y^2-\mu}\right).
 \end{equation}
 When the $\mu$-Markoff map $\phi$ satisfies the Bowditch conditions, $\phi(\vec{e})$ is well defined
 since we have $x,y,z \neq 0$ and (by \cite[Lemma 3.10]{tan-wong-zhang2008advm}) $x,y,z \neq \pm \sqrt{\mu}$.
 By the edge and vertex relations, we have, for any directed edge $\vec{e}\in {\vec E}(\Sigma)$,
 \begin{equation}\label{eqn:phie2sum}
 \phi(\vec{e}) + \phi(-\vec{e})= 1,
 \end{equation}
 and, for any vertex $v \in V(\Sigma)$,
 \begin{equation}\label{eqn:phie3sum}
 \phi(\vec{e_1}) + \phi(\vec{e_2}) + \phi(\vec{e_3}) = 1,
 \end{equation}
 where $\vec{e}_1,\vec{e}_2,\vec{e}_3$ are the three directed edges heading towards $v$.
 Combining relations (\ref{eqn:phie2sum}) and (\ref{eqn:phie3sum}), we have, for any circular set $C=C(\Sigma')$,
 \begin{equation}\label{eqn:phiecircular}
 \sum_{\vec e \in C}\phi(\vec e)=1.
 \end{equation}

 There is also the following branch version of the main theorem.

 \begin{thm}\label{thm:mainbranch}
 Under the same assumptions of Theorem \ref{thm:main} and with the same notation, we have, for every directed edge $\vec{e} \in \vec{E}(\Sigma)$,
 \begin{equation}\label{eqn:mainbranch}
 \sum_{X\in\Omega^0(e)} h_{\mu}(x) + \sum_{X\in\Omega^-(\vec{e})} 2h_{\mu}(x) = \phi(\vec{e}),
 \end{equation}
 where the infinite sum converges absolutely, with edge value $\phi(\vec{e})$ defined by (\ref{eqn:phivece}).
 \end{thm}

 \section{{\bf Proofs of the theorems}}\label{s:proofs}

 \begin{proof}[Proof of Theorem \ref{thm:main}]
 For the $\mu$-Markoff map $\phi={\rm tr}\rho: \Omega \rightarrow \mathbf C$, we need to show
 \begin{equation}\label{eqn:newidmarkoff}
 \sum_{X \in \Omega} h_{\mu}(\phi(X)) = \frac{1}{2},
 \end{equation}
 where the infinite sum converges absolutely. By arguments of Bowditch in \cite{bowditch1998plms},
 it is verified in \cite{tan-wong-zhang2008advm} that
 the infinite sum $\sum_{X \in \Omega} |\phi(X)|^{-t}$ converges for all $t > 0$.
 Since $|h_{\mu}(x)| = O(|x|^{-2})$ as $|x| \to +\infty$, the infinite sum in \eqref{eqn:newidmarkoff} converges absolutely.
 It remains to show that the infinite sum is equal to $\frac12$.

 By a similar calculation as in \cite{tan-wong-zhang2008advm}, there is a constant $K=K(\phi)>0$ so that
 \[
 |\Psi(x,y;z)-h_{\mu}(x)| \le K |y|^{-2}
 \]
 for all vertices $v \leftrightarrow \{X,Y,Z\}$ sufficiently far away from a fixed vertex $v_0$ and $z$ is given by
 \begin{equation}\label{eqn:z}
 z = \frac{xy}{2} \left(1- \sqrt{1-4\left(\frac{1}{x^2} + \frac{1}{y^2} - \frac{\mu}{x^2 y^2}\right) }\right).
 \end{equation}

 By arguments in \cite{bowditch1998plms} (as verified in \cite{tan-wong-zhang2008advm} for general $\mu \in {\mathbf C}$),
 there is a finite subtree $\Sigma_0$ of $\Sigma$ such that for edges not in $\Sigma_0$, their $\phi$-induced directions are all directed towards $\Sigma_0$.
 Let $C_n$ be the set of directed edges at a distance $n$ way from $\Sigma_0$.
 Thus $C_n$ is a circular set, and we have $\sum_{\vec{e} \in C_n} \phi(\vec{e}) = 1$.

 Let $\Omega_n$ denote the subset of $\Omega$ such that $X \in \Omega_n$ if and only if $X \in \Omega^0(e)$ for some $\vec{e}\in C_n$.
 Then $\Omega_{n} \subset \Omega_{n+1}$ and $\Omega=\bigcup_{n=0}^{\infty} \Omega_n$. It suffices to show that
 \[
 \lim_{n \to \infty} \sum_{X\in \Omega_n} h_{\mu}(x) = \frac{1}{2}.
 \]
 In fact, for sufficiently large $n$ and each $\vec{e} \in C_{n+1}$, if $\vec{e} \leftrightarrow \{X,Y;\to Z\}$ with $X,Z \in \Omega_n$ and $Y \in \Omega_{n+1} \backslash \Omega_n$, then $z$ is given by (\ref{eqn:z}). Thus we have
 \begin{align*}
 \left| \sum_{X\in \Omega_n} 2h_{\mu}(x) - 1 \,\right|
 &=\left|\sum_{X\in \Omega_n} 2h_{\mu}(x) - \sum_{\vec{e} \in C_{n+1}} \phi(\vec{e})\right| \\
 &\le \sum_{\vec{e} \in C_{n+1}} |h_{\mu}(x) - \phi(\vec{e})| \\
&\hbox{(where $\vec e \leftrightarrow (X,Y; \rightarrow Z)$ with $X,Z \in \Omega_n, Y \in \Omega_{n+1} \setminus \Omega_n$)}\\
 &\le \sum_{Y \in \Omega_{n+1} \backslash \Omega_n} \text{constant}\cdot |y|^{-2} \to 0
 \end{align*}
 as $n \to \infty$. This proves Theorem \ref{thm:main}.
 \end{proof}


 \begin{proof}[First proof of Theorem \ref{thm:mainpqr}]
 We use similar notation as that used in the proof of Theorem \ref{thm:main}.
 The convergence of the infinite sums in (\ref{eqn:mainpqr}) follows as before, and we need to prove the equality in (\ref{eqn:mainpqr}).
 The idea is to modify the value $\phi(\vec{e})$ for $\vec{e} \in \vec E(\Sigma)$ to a weighted one, $\phi_{p_1,p_2,p_3}(\vec{e})$,
 according to the induced tri-coloring on $\vec E(\Sigma)$ so that they still satisfy the relations (\ref{eqn:phie2sum}) and (\ref{eqn:phie3sum}).

 For any vertex $v \leftrightarrow \{X,Y,Z\}$ of $\Sigma$, where $X \in \Omega_1$, $Y \in \Omega_2$, $Z \in \Omega_3$,
 we denote $\vec{e}_1 \leftrightarrow \{Y,Z;\to X\}$, $\vec{e}_2 \leftrightarrow \{Z,X;\to Y\}$, $\vec{e}_3 \leftrightarrow \{X,Y;\to Z\}$
 and define weighted edge values $\phi_{p_1,p_2,p_3}(\vec{e}_i)$, $i=1,2,3$ by
 \begin{eqnarray}
 \phi_{p_1,p_2,p_3}(\vec{e}_1)&=&\frac{x}{yz}-\left(\frac{1}{2} - \frac{x}{yz}\right)\left(p_2\frac{\mu}{y^2-\mu} + p_3\frac{\mu}{z^2-\mu}\right), \\
 \phi_{p_1,p_2,p_3}(\vec{e}_2)&=&\frac{y}{zx}-\left(\frac{1}{2} - \frac{y}{zx}\right)\left(p_3\frac{\mu}{z^2-\mu} + p_1\frac{\mu}{x^2-\mu}\right), \\
 \phi_{p_1,p_2,p_3}(\vec{e}_3)&=&\frac{z}{xy}-\left(\frac{1}{2} - \frac{z}{xy}\right)\left(p_1\frac{\mu}{x^2-\mu} + p_2\frac{\mu}{y^2-\mu}\right).
 \end{eqnarray}
 It can be checked that 
 \[
 \phi_{p_1,p_2,p_3}(\vec{e_1}) +  \phi_{p_1,p_2,p_3}(\vec{e_2}) +  \phi_{p_1,p_2,p_3}(\vec{e_3}) = 1,
 \]
 and for every directed edge $\vec{e} \in \vec{E}(\Sigma)$,
 \[
 \phi_{p_1,p_2,p_3}(\vec{e}) +  \phi_{p_1,p_2,p_3}(-\vec{e})= 1.
 \]
 The rest of the proof follows similar lines as those in the proof of Theorem \ref{thm:main}.
 \end{proof}


 \begin{proof}[Proof of Theorem \ref{thm:mainbranch}]
 The proof is essentially the same as that of Theorem \ref{thm:main}.
 \end{proof}


 \begin{proof}[Proof of Theorem \ref{thm:relmain}]
 We may deduce Theorem \ref{thm:relmain} from Theorem \ref{thm:mainbranch} with exactly the same arguments as those used
 in \cite[\S 3]{bowditch1997t} or \cite[\S 5]{tan-wong-zhang2008advm}.
 \end{proof}


 \begin{proof}[Second proof of Theorem \ref{thm:mainpqr}]
 With similar notation as that used in the first proof of Theorem \ref{thm:mainpqr},
 let us write, for a directed edge $\vec{e}\leftrightarrow\{X,Y;\rightarrow Z\}$,
 \begin{equation}
 \phi_0(\vec{e})=\frac{z}{xy}.
 \end{equation}
 For each integer $n \ge 0$, let $\Sigma_n \subset \Sigma$ be the subtree whose vertex set consists of vertices
 at a distance not exceeding $n$ from a fixed vertex $v_0$ of $\Sigma$,
 and let $C_n = C(\Sigma_n) \subset \vec{E}(\Sigma)$ be the circular set of directed edges determined by the finite subtree $\Sigma_n$.
 From the vertex and edge relations for the $\mu$-Markoff map $\phi$ we obtain
 \begin{equation}
 \sum_{\vec{e}\in C_n} \phi_0(\vec{e}) - \sum_{\{X,Y,Z\}\leftrightarrow v\in V(\Sigma_n)} \frac{\mu}{xyz}= 1
 \end{equation}
 for all $n \ge 0$. Passing to limits as $n\rightarrow\infty$, we get
 \begin{equation}
 \lim_{n\rightarrow\infty}\sum_{\vec{e}\in C_n} \phi_0(\vec{e}) - \sum_{\{X,Y,Z\}\leftrightarrow v\in V(\Sigma)} \frac{\mu}{xyz}= 1.
 \end{equation}
 By arguments of Bowditch in \cite{bowditch1998plms}, we have
 \begin{equation}
 \lim_{n\rightarrow\infty}\sum_{\vec{e}\in C_n} \phi_0(\vec{e}) = \sum_{X \in \Omega} h_0(x).
 \end{equation}
 On the other hand, for each $v\in V(\Sigma)$ with $v \leftrightarrow \{X_1,X_2,X_3\}$, where $X_i \in \Omega_i$, $i=1,2,3$,
 we may distribute $\frac{\mu}{x_1x_2x_3}$ to the three regions $X_1,X_2,X_3$ according to the weights $p_1, p_2, p_3$ respectively.
 Furthermore, for any $X \in \Omega$ let $V_X \subset V(\Sigma)$ denote the set of vertices adjacent to $X$.
 Then we have
\begin{equation}
 \sum_{\{X,Y,Z\}\leftrightarrow v\in V(\Sigma)} \frac{\mu}{xyz}
  =\sum_{i=1}^{3} p_i\sum_{X\in\Omega_i}\sum_{\{X,Y,Z\}\leftrightarrow v\in V_X} \frac{\mu}{xyz}. \label{eqn:pi}
 \end{equation}
 If a region $X \in \Omega$ corresponds to $a\in \pi_1(T)$, then all the neighboring regions of $X$ are exactly those regions
 $Y_n \in \Omega$, $n \in \mathbf Z$ corresponding to $a^nb \in \pi_1(T)$ where $\{a,b\}$ is free basis of $\pi_1(T)$.
 Thus the values $y_n = \phi(Y_n)$, $n \in \mathbf Z$ are given by the formulas $y_n=A\lambda^n+B\lambda^{-n}$ for some
 $\lambda, A,B \in \mathbf C$ where $\lambda+\lambda^{-1}=x$, $|\lambda|>1$. By the vertex relation, we have
 \begin{equation}\label{eqn:AB}
 AB=\frac{x^2-\mu}{x^2-4}.
 \end{equation}
 Solving for $\lambda$ gives $\lambda=\frac{x}{2}\big(1+\sqrt{1-\frac{4}{x^2}}\big)$, $\lambda^{-1}=\frac{x}{2}\big(1-\sqrt{1-\frac{4}{x^2}}\big)$, and
 \begin{equation}\label{eqn:lambda}
 \lambda-\lambda^{-1} = x\sqrt{1-\frac{4}{x^2}}.
 \end{equation}
 By identity (\ref{eqn:sum}) in Proposition \ref{prop:sums}, we have, for a fixed $X \in \Omega$,
 \begin{equation}
 \sum_{\{X,Y,Z\}\leftrightarrow v\in V_X} \frac{\mu}{xyz} = \frac{\mu}{x^2-\mu}\sqrt{1-\frac{4}{x^2}}.
 \end{equation}
 It follows from (\ref{eqn:pi}) that
 \begin{equation}
 \sum_{X \in \Omega} h_0(x)  - \sum_{i=1}^{3}\sum_{X\in\Omega_{i}} \frac{p_i\mu}{x^2-\mu} \sqrt{1-\frac{4}{x^2}} = 1,
 \end{equation}
 which can be rewritten as (\ref{eqn:mainpqr}). This proves Theorem \ref{thm:mainpqr}.
 \end{proof}

 \noindent {\bf Remark.}
 The theorems in this paper also hold if we relax the Bowditch conditions by allowing a finite number of $X \in \Omega$ with ${\rm tr}\rho(X)=\pm 2$,
 that is, if we replace the Bowditch condition (i) by (i$'$) ${\rm tr}\rho(X) \not\in (-2,2)\subset{\mathbf R}$ for all $X \in \Omega$.
 The verification is essentially the same as what was done in \cite{tan-wong-zhang2006gd} for the results therein.

 \section{Appendix: Two summation identities}\label{s:appendix}

 In this appendix we prove two identities, (\ref{eqn:sum0}) and (\ref{eqn:sum}) below, concerning summations of
 certain sequences of complex numbers.

 \begin{prop}\label{prop:sums}
 Given $\lambda, A, B \in \mathbf C$ with $|\lambda|>1$ and $A,B \neq 0$, let
 $$ y_n=A\lambda^{n}+B\lambda^{-n} $$
 for $n \in \mathbf Z$.
 Then
 \begin{eqnarray}
 \sum_{n=0}^{+\infty} \frac{1}{y_ny_{n+1}}&=&\frac{1}{A(A+B)(\lambda-\lambda^{-1})}, \label{eqn:sum0} \\
 \sum_{n=-\infty}^{+\infty}\frac{1}{y_ny_{n+1}}&=&\frac{1}{AB(\lambda-\lambda^{-1})}. \label{eqn:sum}
 \end{eqnarray}
 \end{prop}

 \begin{proof}
 We first prove (\ref{eqn:sum0}). Indeed,
 \begin{eqnarray*}
 \sum_{n=0}^{+\infty} \frac{1}{y_ny_{n+1}}
 &=& \sum_{n=0}^{+\infty} \frac{\lambda^{2n+1}}{(A\lambda^{2n+1}+B\lambda)(A\lambda^{2n+1}+B\lambda^{-1})} \\
 &=& \frac{1}{A(\lambda-\lambda^{-1})} \sum_{n=0}^{+\infty}
 \left( \frac{\lambda}{A\lambda^{2n+1}+B\lambda}-\frac{\lambda^{-1}}{A\lambda^{2n+1}+B\lambda^{-1}} \right) \\
 &=& \frac{1}{A(\lambda-\lambda^{-1})} \sum_{n=0}^{+\infty}
 \left( \frac{1}{A\lambda^{2n}+B}-\frac{1}{A\lambda^{2n+2}+B} \right) \\
 &=& \frac{1}{A(A+B)(\lambda-\lambda^{-1})}.
 \end{eqnarray*}
 To prove (\ref{eqn:sum}), we have
 \begin{eqnarray*}
 \sum_{n=-\infty}^{+\infty} \frac{1}{y_ny_{n+1}}
 &=& \sum_{n=0}^{+\infty} \frac{1}{y_ny_{n+1}} + \sum_{n=-\infty}^{-1} \frac{1}{y_{n}y_{n+1}} \\ %
 &=& \sum_{n=0}^{+\infty} \frac{1}{y_ny_{n+1}} + \sum_{m=0}^{+\infty} \frac{1}{(B\lambda^{m}+A\lambda^{-m})(B\lambda^{m+1}+A\lambda^{-m-1})} \\ %
 &=& \frac{1}{A(A+B)(\lambda-\lambda^{-1})} + \frac{1}{B(A+B)(\lambda-\lambda^{-1})} \\ %
 &=& \frac{1}{AB(\lambda-\lambda^{-1})}.
 \end{eqnarray*}
 This finishes the proof of Proposition \ref{prop:sums}.
 \end{proof}



\begin{thebibliography}{99}

\bibitem{bowditch1996blms}
    B. H. Bowditch,
    \emph{A proof of McShane's identity via Markoff triples},
    Bull. London Math. Soc. {\bf 28} (1996), 73--78.
\bibitem{bowditch1997t}
    B. H. Bowditch,
    \emph{A variation of McShane's identity for once punctured torus bundles},
    Topology {\bf 36} (1997), 325--334.
\bibitem{bowditch1998plms}
    B. H. Bowditch,
    \emph{Markoff triples and quasi-Fuchsian groups},
    Proc. London Math. Soc. (3) {\bf 77} (1998), 697--736.
\bibitem{goldman2003gt}
    W. M. Goldman,
    \emph{The modular group action on real ${\rm SL}(2)$-characters of a one-holed torus},
    Geom. Topol. {\bf 7} (2003), 443--486.
\bibitem{hu2013thesis}
    H. Hu,
    \emph{Identities on hyperbolic surfaces, Group actions and the Markoff-Hurwitz equations},
    Ph.D. Thesis, National University of Singapore, 2013.
\bibitem{hu-tan-zhang2}
    H. Hu, S. P. Tan and Y. Zhang,
    \emph{Polynomial automorphisms of ${\mathbf C}^n$ preserving the Markoff-Hurwitz equation},
    in preparation.
\bibitem{luo-tan}
    F. Luo and S. P. Tan,
    \emph{A dilogarithm identity on moduli spaces of curves},
    J. Differential Geom., to appear.
\bibitem{mcshane1991thesis}
    G. McShane,
    \emph{A remarkable identity for lengths of curves},
    Ph.D. Thesis, University of Warwick, 1991.
\bibitem{mcshane1998im}
    G. McShane,
    \emph{Simple geodesics and a series constant over Teichmuller space},
    Invent. math. {\bf 132} (1998), 607--632.
\bibitem{mirzakhani2007im}
    M. Mirzakhani,
    \emph{Simple geodesics and Weil-Petersson volumes of moduli spaces of bordered Riemann surfaces},
    Invent. math. {\bf 167} (2007), 179--222.
\bibitem{rivin2012advm}
    I. Rivin,
    \emph{Geodesics with one self-intersection, and other stories},
    Adv. Math. {\bf 231} (2012), 2391--2412.
\bibitem{tan-wong-zhang2006jdg}
    S. P. Tan, Y. L. Wong and Y. Zhang,
    \emph{Generalizations of McShane's identity to hyperbolic cone-surfaces},
    J. Differential Geom. {\bf 72} (2006), 73--112.
\bibitem{tan-wong-zhang2006gd}
    S. P. Tan, Y. L. Wong and Y. Zhang,
    \emph{Necessary and sufficient conditions for McShane's identity and variations},
    Geom. Dedicata {\bf 119} (2006), 199--217.
\bibitem{tan-wong-zhang2008advm}
    S. P. Tan, Y. L. Wong and Y. Zhang,
    \emph{Generalized Markoff maps and McShane's identity},
    Adv. Math. {\bf 217} (2008), 761--813.
\end{thebibliography}
\end{document}